\documentclass{amsart}

\usepackage{amsfonts, color}
\usepackage{latexsym}
\usepackage{amssymb}
\usepackage{amsmath}
\usepackage{amsthm}
\usepackage{epsfig}
\usepackage{graphicx,colortbl}
\usepackage{url,hyperref}
\usepackage{mathtools, color, changes}

\newcommand{\R}{\mathbb R}

\newtheorem{thm}{Theorem}[section]

\newtheorem{lemma}{Lemma}[section]

\newtheorem{rmk}{Remark}[section]

\numberwithin{equation}{section}

\newcommand{\tg}{\widetilde g_{K,L}}

\hypersetup{hidelinks}

\begin{document}

\title[Improved Brunn--Minkowski inequality for $\theta$-convolution bodies]
{Improved Brunn--Minkowski inequality for $\theta$-convolution bodies}

\author{Javier Martín Goñi}
\address{Faculty of Computer Science and Mathematics, University of Passau, Innstrasse 33, 94032 Passau, Germany}
\email{j.martin@unizar.es}

\thanks{Supported by Project
PID2025-169474NB-C22  funded by  MICIU/AEI/10.13039/501100011033/ and partially supported by DFG project 516672205 and by DGA project E48\_26R}

\begin{abstract}
For $n\geq 2$, we study the optimal function $\varphi_n$ in the Brunn--Minkowski
inequality for $\theta$-convolution bodies
\[
 |K+_\theta L|^{1/n}\ge
 \varphi_n(\theta)\bigl(|K|^{1/n}+|L|^{1/n}\bigr)
\]
for any convex bodies $K,L \subset \R^n$. We improve the known lower bound for every $0<\theta<(3/4)^n$
and determine $\varphi_n$ on an interval strictly containing
$[(2/3)^n,1]$. For $\theta<1$ in this interval, equality holds if and
only if $K$ and $-L$ are translates of the same
$n$-dimensional simplex. The proof reduces the geometric
problem to a sharp one-dimensional estimate for nonnegative concave
functions under a weighted integral constraint.
\end{abstract}

\date{\today}
\maketitle

\section{Introduction and main results}

Let $K,L\subseteq\R^n$ be convex bodies, and let $|\cdot|$ denote
$n$-dimensional Lebesgue measure. The Brunn--Minkowski inequality
gives the sharp lower bound
\begin{equation}\label{eq:BM-classical-intro}
 |K+L|^{1/n}\ge |K|^{1/n}+|L|^{1/n}
\end{equation}
for the volume of their Minkowski sum
$K+L=\{x+y:x\in K,\ y\in L\} = \{z \in \R^n : K\cap (z-L) \neq \emptyset\}$, where equality holds if and only if \(K\) and \(L\) are homothetic; see
\cite[Theorem 7.1.1]{Sch}. The $\theta$-convolution bodies
studied in \cite{AJV} restrict this sum to translations
for which the intersection volume is at least a prescribed
fraction of its maximum. More precisely, their
$\theta$-convolution body is
\begin{equation}\label{eq:convolution-body}
 K+_\theta L
 =\{z\in K+L:|K\cap(z-L)|\ge\theta M(K,L)\},
 \qquad 0\le\theta\le1,
\end{equation}
where
\[
 M(K,L)=\max_{z\in\R^n}|K\cap(z-L)|.
\]
In particular, $K+_0L=K+L$. These bodies interpolate between the
Minkowski sum and the set of translations at which the intersection
volume is maximal; see \cite{AJV,AGMG}. Earlier work on convolution level sets of symmetric convex
bodies includes \cite{Schm,Tso}.

Following \cite{AJV,AGMG}, we consider the largest function
$\varphi_n:[0,1]\to\R$ such that
\begin{equation}\label{eq:BM}
 |K+_\theta L|^{1/n}\ge
 \varphi_n(\theta)\bigl(|K|^{1/n}+|L|^{1/n}\bigr)
\end{equation}
for every pair of convex bodies. Equivalently,
\begin{equation}\label{eq:optimal-function}
 \varphi_n(\theta)=\inf_{K,L}
 \frac{|K+_\theta L|^{1/n}}{|K|^{1/n}+|L|^{1/n}}.
\end{equation}
The Brunn--Minkowski inequality \eqref{eq:BM-classical-intro} and its equality cases give $\varphi_n(0)=1$, and
opposite convex bodies $K=-L$ give $\varphi_n(1)=0$. We assume $n\ge2$.


\medskip

The first known bound for $\varphi _n$ was given by Alonso-Guti\'errez, Jim\'enez and Villa, where they established
in \cite[Corollary 3.2]{AJV} that
\[
 \varphi_n(\theta)\ge 1-\theta^{1/n},
 \qquad 0\le\theta\le1.
\]
Later, Alonso-Guti\'errez and Mart\'in Go\~ni improved this bound, by proving in
\cite[Theorem 1.1]{AGMG} that
\begin{equation}\label{eq:AGMG}
 \varphi_n(\theta)\ge
 \begin{cases}
  1-\left(\dfrac43-\dfrac16\binom{2n}{n}^{1/n}\right)\theta^{1/n},
       &0\le\theta\le(3/4)^n,\\[4pt]
  \dfrac12\binom{2n}{n}^{1/n}(1-\theta^{1/n}),
       &(3/4)^n\le\theta\le1.
 \end{cases}
\end{equation}
The second expression is optimal in its stated interval, and its proof uses an inclusion relating Ball's bodies of the function $g_{K,L}(x)=|K\cap(x-L)|$ to its superlevel sets. 
The first is obtained by interpolation, in the variable
$\theta^{1/n}$, between Brunn--Minkowski and the sharp value at
$\theta=(3/4)^n$. As noted in \cite{AGMG}, the level $(3/4)^n$ comes
from the argument used there. This raises the question of finding the largest interval $[\theta _0 , 1]$ in which the second expression in \eqref{eq:AGMG} is sharp, and therefore the ratio in \eqref{eq:optimal-function} is minimized by opposite simplices.

\medskip

Our bound determines $\varphi_n$ exactly on an interval
strictly containing $[(2/3)^n,1]$, extending the interval
$[(3/4)^n,1]$ obtained in \cite{AGMG}.
For $\theta<1$ in this larger interval, equality holds
if and only if $K$ and $-L$ are translates of the same
simplex. We also strictly improve the lower bound of $\varphi _n$
in \eqref{eq:AGMG} for every $0<\theta<(3/4)^n$.

For $n\ge2$, let $\beta_n$ be the unique number in $(0,1)$ satisfying
\begin{equation}\label{eq:beta-definition}
 n\int_0^1t^{n-1}(1-\beta_n t)^n\,\mathrm dt=2^{-n}.
\end{equation}
Existence and uniqueness are verified in Section~\ref{sec:proof}.

\begin{thm}\label{thm:main}
Let $n\ge2$. For every $0\le\theta\le1$,
\begin{equation}\label{eq:new-bound}
 \varphi_n(\theta)\ge
 \min\left\{
  1-\beta_n\theta^{1/n},\,
  \frac12\binom{2n}{n}^{1/n}(1-\theta^{1/n})
 \right\}.
\end{equation}
Moreover,
\begin{equation}\label{eq:exact-value}
 \varphi_n(\theta)=
 \frac12\binom{2n}{n}^{1/n}(1-\theta^{1/n})
\end{equation}
whenever $\theta \in [\theta _0 , 1]$, with
\begin{equation}\label{eq:sharp-range}
 \theta _0 = \left(
 \frac{\binom{2n}{n}^{1/n}-2}{\binom{2n}{n}^{1/n}-2\beta_n}
 \right)^n .
\end{equation}
For every $\theta \in [\theta _0 , 1)$, equality in \eqref{eq:BM} holds if and only if $K$ and $-L$
are translates of the same $n$-dimensional simplex.
\end{thm}

\begin{rmk}
The coefficient in \eqref{eq:beta-definition} satisfies (see Section \ref{sec:comparison} below)
\[
 0<\beta_n<2\binom{2n}{n}^{-1/n}.
\]
Since $2<\binom{2n}{n}^{1/n}<4$, then
\[
 \frac{\binom{2n}{n}^{1/n}-2}
      {\binom{2n}{n}^{1/n}-2\beta_n}
 <\frac{\binom{2n}{n}^{1/n}}{\binom{2n}{n}^{1/n}+2}
 <\frac23.
\]
Raising to the $n$th power shows that $\theta _0$ is strictly smaller than $(2/3)^n$.
Consequently, Theorem~\ref{thm:main} implies
\[
 \varphi_n(\theta)=\frac12\binom{2n}{n}^{1/n}
 (1-\theta^{1/n}),\qquad (2/3)^n\le\theta\le1.
\]
\end{rmk}

The extension of the exact
interval is already substantial in dimension two: in this case,
\eqref{eq:beta-definition} gives
\[
 \beta_2=\frac{8-\sqrt{10}}6,
\]
so we obtain the exact value of $\varphi_2$ in the interval
$(0.2885,1)$, whereas the previous known sharp range in \cite{AGMG} begins at
$9/16=0.5625$. In our interval, the optimal coefficient is therefore
determined exactly as
\[
 \varphi_2(\theta)=\sqrt{\frac32}\,(1-\sqrt{\theta}),
\]
with equality precisely for opposite translates of the same
triangle.

\medskip

The proof of Theorem~\ref{thm:main} relies on the following sharp inequality
for nonnegative concave functions. For each $s\in[0,1]$,
the theorem gives the smallest possible value of $H(s)$
under its assumptions. This value is attained by one
of the two affine functions in the statement.

\begin{thm}\label{lem:scalar}
Let $n\ge2$ be an integer, and let $\beta_n$ be defined by
\eqref{eq:beta-definition}. Let $H:[0,1]\to[0,\infty)$ be
continuous and concave. Assume that $H(0)\ge1$ and
\begin{equation}\label{eq:scalar-mass}
 n\int_0^1t^{n-1}H(t)^n\,\mathrm dt\ge2^{-n}.
\end{equation}
Then
\begin{equation}\label{eq:scalar-bound}
 H(s)\ge\min\left\{1-\beta_ns,
 \frac12\binom{2n}{n}^{1/n}(1-s)\right\},
 \qquad 0\le s\le1.
\end{equation}
This bound is optimal at every $s$ under the stated hypotheses.
If equality holds in \eqref{eq:scalar-bound} at some $0<s<1$,
then either
\[
 H(t)=1-\beta_nt,\quad\text{for all }t\in[0,1],
\]
or
\[
 H(t)=\frac12\binom{2n}{n}^{1/n}(1-t),
 \quad\text{for all }t\in[0,1].
\]
In either case equality holds in \eqref{eq:scalar-mass}.
\end{thm}

For a pair of convex bodies, we apply this theorem to the
normalized volume profile
\[
 H(s)=\frac{|K+_{s^n}L|^{1/n}}
 {|K|^{1/n}+|L|^{1/n}},\qquad 0\le s\le1.
\]
The level-set inclusion in \cite[Proposition 2.2]{AJV}, together
with Brunn--Minkowski, gives concavity and $H(0)\ge1$, while
the convolution integral identity and the elementary bound
$M(K,L)\le\min\{|K|,|L|\}$ yield \eqref{eq:scalar-mass};
see also \cite[proof of Theorem 1.1]{AGMG}.
These properties, including the equality case in the integral
bound, are collected in Lemma~\ref{lem:volume-profile}.
Theorem~\ref{lem:scalar} gives \eqref{eq:new-bound}.
On the interval $[\theta _0 , 1]$, this lower bound
coincides with the value attained by opposite simplices,
so it determines $\varphi_n$ exactly.
If equality holds for $\theta<1$ in this interval,
equality in the integral bound first forces $K$ and
$-L$ to be translates. The equality case of the
Rogers--Shephard inequality \cite{RS} then shows
that they are simplices.

Since \eqref{eq:scalar-bound} is optimal under the
hypotheses of Theorem~\ref{lem:scalar}, any stronger
bound for $\varphi_n$ would have to use further
properties of the volume profiles.

\medskip

The structure of this work is the following. Section~\ref{sec:preliminaries} recalls the classical results
used in the proof. Section~\ref{sec:proof} establishes the volume
profile properties and the scalar estimate, and proves
Theorem~\ref{thm:main}, including its equality cases.
Section~\ref{sec:comparison} explains the role of $\beta_n$,
compares our bound with \eqref{eq:AGMG}, and uses Euclidean balls
to limit how far the interval of optimality of opposite
simplices could extend towards zero.

\section{Preliminaries}\label{sec:preliminaries}

We recall the notation and the classical facts used in the proof.
A convex body is a compact convex subset of $\R^n$ with nonempty
interior, and $|\cdot|$ always denotes $n$-dimensional Lebesgue
measure. We write $K-K=K+(-K)$ for the difference body of $K$.

For convex bodies $K,L$, the generalized covariogram is the
convolution of their indicator functions:
\[
 |K\cap(z-L)|=(\mathbf1_K*\mathbf1_L)(z),\qquad z\in\R^n;
\]
see \cite[Section 2.3]{AGMG}. This function is continuous, has
support $K+L$, and is positive precisely on
$\operatorname{int}(K+L)$. In particular, its maximum $M(K,L)$
is positive and attained. For a fixed pair $K,L$, we write
\[
 \tg(z)=\frac{|K\cap(z-L)|}{M(K,L)}.
\]
By the Brunn–Minkowski inequality, \(\widetilde g_{K,L}^{\,1/n}\) is concave on \(K+L\). We denote the superlevel sets of \(\widetilde g_{K,L}\) by
\[
 L_s=\{z\in K+L:\tg(z)\ge s^n\}
     =K+_{s^n}L,\qquad 0\le s\le1.
\]
These sets are nonempty and compact. We have $L_0=K+L$,
whereas $L_1$ consists of the points at which the maximum is
attained and may have zero volume.

Translations of $K$ and $L$ leave $M(K,L)$ unchanged and satisfy
\[
 (K+a)+_\theta(L+b)=a+b+(K+_\theta L),
 \qquad a,b\in\R^n,\quad 0\le\theta\le1.
\]
Thus they do not change the ratio defining $\varphi_n$.

We use the Brunn--Minkowski inequality in its concavity form:
for nonempty compact convex sets $A,B$ and $0\le\lambda\le1$,
\begin{equation}\label{eq:classical-BM}
 |(1-\lambda)A+\lambda B|^{1/n}
 \ge(1-\lambda)|A|^{1/n}+\lambda|B|^{1/n};
\end{equation}
see \cite[Theorem 7.1.1]{Sch}. The inequality also applies when
one of the sets has zero volume, which is needed for the
maximum level $L_1$.

To identify the equality cases, we use the Rogers--Shephard
inequality \cite{RS}:
\begin{equation}\label{eq:RS}
 |K-K|\le\binom{2n}{n}|K|,
\end{equation}
with equality if and only if $K$ is an $n$-dimensional simplex.
For such a simplex $\Delta$, the convolution bodies have the
explicit form
\begin{equation}\label{eq:simplex-levels}
 \Delta+_{s^n}(-\Delta)=(1-s)(\Delta-\Delta),
 \qquad 0\le s\le1;
\end{equation}
see \cite[Proposition 2.8]{AJV} or
\cite[Section 2.3]{AGMG}. Hence equality in \eqref{eq:RS} gives
\[
 \frac{|\Delta+_{s^n}(-\Delta)|^{1/n}}{2|\Delta|^{1/n}}
 =\frac{(1-s)|\Delta-\Delta|^{1/n}}{2|\Delta|^{1/n}}
 =\frac12\binom{2n}{n}^{1/n}(1-s).
\]
Since $\varphi_n$ is the infimum of the corresponding ratios over
all pairs of convex bodies, this yields
\begin{equation}\label{eq:simplex-upper}
 \varphi_n(s^n)\le
 \frac12\binom{2n}{n}^{1/n}(1-s),\qquad 0\le s\le1.
\end{equation}
Thus proving the reverse inequality on an interval determines
$\varphi_n$ there, with opposite simplices attaining its value.

Finally, we recall two integral identities. The beta integral
\begin{equation}\label{eq:beta-value}
 n\int_0^1 t^{n-1}(1-t)^n\,\mathrm dt
 =\binom{2n}{n}^{-1}
\end{equation}
which will be used to
integrate the affine functions in the scalar estimate, and the layer-cake identity
\begin{equation}\label{eq:layer-cake}
 \int_{\R^n}f(z)\,\mathrm dz
 =\int_0^1|\{z:f(z)\ge t\}|\,\mathrm dt
\end{equation}
for every nonnegative measurable function $f\le1$ with bounded
support. It follows from Tonelli's theorem applied to
$f(z)=\int_0^1\mathbf1_{\{f(z)\ge t\}}\,\mathrm dt$.
Applied to $f=\tg$, it relates the integral of the normalized
covariogram to the volumes of the convolution bodies.

\section{Proof of the main results}\label{sec:proof}

\subsection{Bound for the integral}

The following lemma collects standard properties of the
volume profile. The level-set inclusion and the integral
identity appear in \cite[Proposition 2.2]{AJV} and
\cite[proof of Theorem 4.6]{AJV}, respectively.
The mass lower bound in \eqref{eq:profile-mass} below is also used, in an equivalent
form, in \cite[proof of Theorem 1.1]{AGMG}.
We include the details, including the equality case,
for completeness.

\begin{samepage}
\begin{lemma}\label{lem:volume-profile}
Let $K,L\subseteq\R^n$ be convex bodies, and set
\[
 S=|K|^{1/n}+|L|^{1/n},\qquad
 H(s)=\frac{|L_s|^{1/n}}{S},\qquad 0\le s\le1.
\]
Then $H$ is continuous, nonnegative and concave,
$H(0)\ge1$, and
\begin{equation}\label{eq:profile-mass}
 n\int_0^1s^{n-1}H(s)^n\,\mathrm ds
 =\frac{|K||L|}{M(K,L)S^n}\ge2^{-n}.
\end{equation}
Equality in the last inequality holds if and only if $K$ and
$-L$ are translates.
\end{lemma}
\end{samepage}

\begin{proof}
For $z,w\in K+L$ and $0\le\lambda\le1$, convexity of $K$ and $L$
gives
\[
 (1-\lambda)(K\cap(z-L))+\lambda(K\cap(w-L))
 \subseteq K\cap\bigl((1-\lambda)z+\lambda w-L\bigr).
\]
Taking the $n$th root of the volumes, applying Brunn--Minkowski inequality
\eqref{eq:classical-BM} to the intersections, and dividing by
$M(K,L)^{1/n}$, we obtain
\begin{equation}\label{eq:overlap-concavity}
 \tg((1-\lambda)z+\lambda w)^{1/n}
 \ge(1-\lambda)\tg(z)^{1/n}+\lambda\tg(w)^{1/n}.
\end{equation}

Let $r,s\in[0,1]$, $z\in L_r$ and $w\in L_s$.
By the definition of the superlevel sets in
Section~\ref{sec:preliminaries},
$\tg(z)^{1/n}\ge r$ and $\tg(w)^{1/n}\ge s$.
Thus \eqref{eq:overlap-concavity} implies
\[
 \tg((1-\lambda)z+\lambda w)^{1/n}
 \ge(1-\lambda)r+\lambda s.
\]
Equivalently,
\begin{equation}\label{eq:level-inclusion}
 (1-\lambda)L_r+\lambda L_s
 \subseteq L_{(1-\lambda)r+\lambda s}.
\end{equation}
Taking $r=s$ shows that each $L_s$ is convex.

We can therefore apply Brunn--Minkowski inequality \eqref{eq:classical-BM} to
$L_r$ and $L_s$. Together with \eqref{eq:level-inclusion},
this gives
\begin{align*}
 H((1-\lambda)r+\lambda s)
 &=\frac{|L_{(1-\lambda)r+\lambda s}|^{1/n}}{S}\\
 &\ge\frac{|(1-\lambda)L_r+\lambda L_s|^{1/n}}{S}\\
 &\ge\frac{(1-\lambda)|L_r|^{1/n}+\lambda|L_s|^{1/n}}{S}\\
 &=(1-\lambda)H(r)+\lambda H(s).
\end{align*}
This proves the concavity of $H$.
Note that nonnegativity follows from its definition. Since $L_0=K+L$,
Brunn--Minkowski also gives
\[
 H(0)=\frac{|K+L|^{1/n}}{|K|^{1/n}+|L|^{1/n}}\ge1.
\]

Note that concavity implies continuity on $(0,1)$. For the endpoint limits, observe that if $0\le r\le s\le1$,
then $\tg(z)\ge s^n$ implies $\tg(z)\ge r^n$, and hence
$L_s\subseteq L_r$.

Let $(s_j)_{j=1}^\infty$ be an increasing sequence with $s_j>0$ and $\lim _{j \to \infty} s_j = 1$. Then,
$L_{s_{j+1}}\subseteq L_{s_j}$ and
$\bigcap_jL_{s_j}=L_1$.
Since these sets have finite volume, continuity of measure from
above gives
\[
 \lim_{j\to\infty}|L_{s_j}|=|L_1|,
 \qquad \lim_{j\to\infty}H(s_j)=H(1).
\]

If $(s_j)_{j=1}^\infty$ is a decreasing sequence with $s_j>0$ and $\lim _{j \to \infty} s_j = 0$, then
$L_{s_j}\subseteq L_{s_{j+1}}$ and
\[
 \bigcup_jL_{s_j}=\{z:\tg(z)>0\}
 =\operatorname{int}(K+L),
\]
since $\tilde{g}_{K,L}(z) >0$ for any $z \in \operatorname{int}(K + L)$.

The boundary of a convex body has zero volume, so continuity
of measure from below yields
\[
 \lim_{j\to\infty}|L_{s_j}|=|K+L|=|L_0|,
 \qquad \lim_{j\to\infty}H(s_j)=H(0).
\]
This proves continuity on $[0,1]$, as required in
Theorem~\ref{lem:scalar}.


\quad

We next integrate the volumes in the original parameter $\theta$.
For $0<\theta\le1$, the set
$\{z:\tg(z)\ge\theta\}$ is exactly $K+_\theta L$.
Therefore \eqref{eq:layer-cake} gives
\begin{equation}
\label{eq:layer-cake-profile}
 \int_0^1|K+_\theta L|\,\mathrm d\theta
 =\int_{\R^n}\tg(z)\,\mathrm dz =\frac1{M(K,L)}\int_{\R^n}|K\cap(z-L)|\,\mathrm dz.
\end{equation}

Where, by Tonelli's theorem,
\begin{align*}
 \int_{\R^n}|K\cap(z-L)|\,\mathrm dz = \int_{\R^n}\int_{\R^n}
      \mathbf1_K(x)\mathbf1_L(z-x)\,\mathrm dx\,\mathrm dz =|K||L|.
\end{align*}
Substituting $\theta=s^n$ in \eqref{eq:layer-cake-profile} and
using $|K+_{s^n}L|=|L_s|=S^nH(s)^n$, we obtain
\[
 n\int_0^1s^{n-1}H(s)^n\,\mathrm ds
 =\frac1{S^n}\int_0^1|K+_\theta L|\,\mathrm d\theta
 =\frac{|K||L|}{M(K,L)S^n}.
\]
This establishes the identity in \eqref{eq:profile-mass}.

To prove its lower bound, assume without loss of generality that
$|K|\ge|L|$. Since $M(K,L)\le|L|$,
\begin{equation}\label{eq:mass-lower-bound}
 \frac{|K||L|}{M(K,L)S^n}
 \ge\frac{|K|}{(|K|^{1/n}+|L|^{1/n})^n}
 =\left(\frac1{1+(|L|/|K|)^{1/n}}\right)^n
 \ge2^{-n}.
\end{equation}
Equality in the first inequality is equivalent to
$M(K,L)=|L|$, and equality in the last is equivalent to
$|K|=|L|$. If both hold, choose $z_0$ such that
$|K\cap(z_0-L)|=M(K,L)$. Then
\[
 |K\cap(z_0-L)|=|K|=|z_0-L|.
\]
The intersection has positive volume and is therefore
itself a convex body. Since a proper inclusion between
convex bodies gives a strict inequality of volumes,
both inclusions must be equalities. Hence
\[
 K=K\cap(z_0-L)=z_0-L.
\]
Conversely, if $K=z_0-L$, then
$M(K,L)=|K|=|L|$ and $S^n=2^n|K|$, which gives equality
in \eqref{eq:profile-mass}.
\end{proof}

\subsection{An estimate using supporting affine functions}

We now prove Theorem~\ref{lem:scalar}. To explain the
argument, fix $s\in(0,1)$ and write $y=H(s)$. By concavity, $H$ admits
a supporting affine function $h_a(t)=y+a(s-t)$, with
$h_a\ge H$ on $[0,1]$, where $-a$ is the slope.
The endpoint conditions $H(0)\ge1$ and $H(1)\ge0$ restrict
the parameter $a$ to a closed interval. Its endpoints
correspond to the affine functions through $(s,y)$ that
satisfy $h_a(0)=1$ or $h_a(1)=0$, respectively. For fixed $s$ and $y$, the weighted integral
\[
 n\int_0^1 t^{n-1}h_a(t)^n\,\mathrm dt
\]
is convex in $a$ on this interval, and is therefore bounded
above by the larger of its two endpoint values. Since the chosen supporting function lies above $H$,
at least one of these two endpoint integrals must be
at least $2^{-n}$. Evaluating them yields the two
expressions in \eqref{eq:scalar-bound}.

\medskip

For the rest of the proof, write
\begin{equation}\label{eq:c}
 c=\frac12\binom{2n}{n}^{1/n}.
\end{equation}
When $n\ge2$, we have $1<c<2$. Indeed,
\[
 \binom{2n}{n}=\prod_{j=1}^n\frac{n+j}{j}>2^n,
 \qquad
 \binom{2n}{n}<\sum_{j=0}^{2n}\binom{2n}{j}=4^n.
\]
The first strict inequality follows because every factor is at
least $2$ and the factors with $j<n$ are strictly greater than $2$.

To justify the definition of $\beta_n$, let us define
\begin{equation}\label{eq:In-definition}
 I_n(x)=n\int_0^1t^{n-1}(1-xt)^n\,\mathrm dt,
 \qquad 0\le x\le1.
\end{equation}
This is a continuous function of $x$, and differentiation gives
\begin{equation}\label{eq:In-derivative}
 I_n'(x)=-n^2\int_0^1t^n(1-xt)^{n-1}\,\mathrm dt<0.
\end{equation}
The integrand is strictly positive for $0<t<1$, so $I_n$
is strictly decreasing on $[0,1]$. Its endpoint values are
\[
 I_n(0)=1,\qquad
 I_n(1)=\binom{2n}{n}^{-1}<2^{-n},
\]
so the intermediate value theorem gives a unique
$\beta_n\in(0,1)$ with $I_n(\beta_n)=2^{-n}$.

\begin{proof}[Proof of Theorem~\ref{lem:scalar}]
At $s=0$, the right-hand side of \eqref{eq:scalar-bound} is
$\min\{1,c\}=1$, and at $s=1$ it is zero, so these cases follow
from the endpoint assumptions.

Fix $s\in(0,1)$ and put $y=H(s)$. If $y\ge1$, the bound holds
because $1-\beta_ns<1$. We may therefore assume $y<1$.
Concavity of $H(s)$ on $[0,1]$ gives
\begin{equation}\label{eq:y-lower}
 y\ge(1-s)H(0)+sH(1)\ge1-s>0.
\end{equation}

Concavity allows us to construct a supporting line of the hypograph of $H$ at $s$, that is, there exists a real number $a$ such that the affine
function
\begin{equation}\label{eq:support-line}
 h_a(t)=y+a(s-t),\qquad 0\le t\le1,
\end{equation}
satisfies
\begin{equation}\label{eq:support-domination}
 h_a(s)=H(s),\qquad H(t)\le h_a(t)\quad(0\le t\le1).
\end{equation}
To see this explicitly, let $H'_+(s)$ and $H'_-(s)$ denote
the right and left derivatives. Choose $a$ so that
$H'_+(s)\le-a\le H'_-(s)$.
Then,
\[
 \begin{aligned}
 \frac{H(s)-H(t)}{s-t}&\ge H'_-(s)\ge-a
       &&(0\le t<s),\\
 \frac{H(t)-H(s)}{t-s}&\le H'_+(s)\le-a
       &&(s<t\le1).
 \end{aligned}
\]
Rearranging the inequalities yields
$H(t)\le H(s)+a(s-t)=h_a(t)$, proving
\eqref{eq:support-domination}.
Evaluating this inequality at $t=0$ and $t=1$ yields
\[
 y+as=h_a(0)\ge H(0)\ge1,\qquad
 y-a(1-s)=h_a(1)\ge H(1)\ge0.
\]

Thus, defining
\begin{equation}\label{eq:admissible-slopes}
 a_-=\frac{1-y}{s},\qquad a_+=\frac{y}{1-s},
\end{equation}
we have $a_-\le a\le a_+$. Notice that
$a_-\le a_+$ is equivalent to $y\ge1-s$, already proved
in \eqref{eq:y-lower}.
In fact $y>1-s$; otherwise $a_-=a_+=1$, so
$H(t)\le1-t$ and
\[
 n\int_0^1t^{n-1}H(t)^n\,\mathrm dt
 \le\binom{2n}{n}^{-1}<2^{-n},
\]
contrary to \eqref{eq:scalar-mass}. Hence $a_-<a_+$.

We now keep $s$ and $y=H(s)$ fixed and allow $a$ to range
over the whole interval $[a_-,a_+]$. The functions in this
enlarged family need not all lie above $H$, but the
parameter of every supporting function of the form
\eqref{eq:support-line} belongs to this interval. 

For every $a\in[a_-,a_+]$, we have $h_a(0)\ge1$ and
$h_a(1)\ge0$. Since $h_a$ is affine, it is nonnegative
on $[0,1]$ and strictly positive on $(0,1)$.
For $a=a_-$ and $a=a_+$, the corresponding functions are
\begin{equation}\label{eq:endpoint-functions}
 \begin{aligned}
  h_-(t)&:=h_{a_-}(t)=1-\frac{1-y}{s}\,t,\\
  h_+(t)&:=h_{a_+}(t)=\frac{y}{1-s}(1-t).
 \end{aligned}
\end{equation}
Both take the value $y$ at $t=s$, and they satisfy
$h_-(0)=1$ and $h_+(1)=0$, respectively.
We will bound the weighted integral of the chosen
supporting function by the larger of the corresponding
integrals for $h_-$ and $h_+$.

Define the function of the parameter $a$
\begin{equation}\label{eq:F-definition}
 F(a)=n\int_0^1t^{n-1}h_a(t)^n\,\mathrm dt,
 \qquad a\in[a_-,a_+].
\end{equation}
Since $n$ is an integer, $F$ is a polynomial in $a$.
Differentiating twice, we find
\begin{equation}\label{eq:F-convexity}
 F''(a)=n^2(n-1)\int_0^1
 t^{n-1}(s-t)^2h_a(t)^{n-2}\,\mathrm dt>0.
\end{equation}
Since $n\ge2$, the factor $t^{n-1}$ is positive for $0<t<1$,
and $h_a(t)>0$ there. Thus the integrand is positive at every
$t\in(0,1)$ except at $t=s$.
For $n=2$, the last factor in the integrand is simply $1$.
This proves strict convexity of $F$.

For any $a\in[a_-,a_+]$, set
$\lambda=(a-a_-)/(a_+-a_-)\in[0,1]$. Then
$a=(1-\lambda)a_-+\lambda a_+$, so convexity gives
\begin{equation}\label{eq:F-endpoint-bound}
 F(a)\le(1-\lambda)F(a_-)+\lambda F(a_+)
 \le\max\{F(a_-),F(a_+)\}.
\end{equation}
The first inequality is strict whenever $a_-<a<a_+$.
Using \eqref{eq:endpoint-functions} and the beta integral \eqref{eq:beta-value},
we obtain the following expressions for $F(a_-)$
and $F(a_+)$:
\begin{align}
 F(a_-)
 &=n\int_0^1t^{n-1}
       \left(1-\frac{1-y}{s}\,t\right)^n\,\mathrm dt
 =I_n\left(\frac{1-y}{s}\right),\label{eq:F-left}\\
 F(a_+)
 &=\left(\frac{y}{1-s}\right)^n
      n\int_0^1t^{n-1}(1-t)^n\,\mathrm dt\notag\\
 &=\frac{y^n}{\binom{2n}{n}(1-s)^n}
 =2^{-n}\left(\frac{y}{c(1-s)}\right)^n.
 \label{eq:F-right}
\end{align}
Note that the argument $(1-y)/s$ in \eqref{eq:F-left} belongs to $(0,1)$
because $1-s<y<1$.

Choose now a parameter $a$ for which \eqref{eq:support-domination}
holds. Combining the assumption on the lower bound on the integral of $H$ \eqref{eq:scalar-mass}, the
inequality $H\le h_a$, and \eqref{eq:F-endpoint-bound}, we obtain
\begin{equation}\label{eq:two-endpoints}
 \begin{aligned}
 2^{-n}
 &\le n\int_0^1t^{n-1}H(t)^n\,\mathrm dt\\
 &\le F(a)\le\max\{F(a_-),F(a_+)\}.
 \end{aligned}
\end{equation}

Now, suppose that
$y<\min\{1-\beta_ns,c(1-s)\}$.
The first inequality, $y<1-\beta_ns$, implies
$(1-y)/s>\beta_n$. Since $I_n$ is strictly decreasing
by \eqref{eq:In-derivative}, \eqref{eq:F-left} gives
\[
 F(a_-)=I_n\left(\frac{1-y}{s}\right)
 <I_n(\beta_n)=2^{-n}.
\]
The second inequality, $y<c(1-s)$, gives, by
\eqref{eq:F-right},
\[
 F(a_+)=2^{-n}\left(\frac{y}{c(1-s)}\right)^n<2^{-n}.
\]
Both $F(a_-)$ and $F(a_+)$ would then be smaller
than $2^{-n}$,
contradicting \eqref{eq:two-endpoints}.
Then, $H(s)\ge\min\{1-\beta_n s,\ c(1-s)\}$, for any $0\le s\le1$, proving \eqref{eq:scalar-bound}.

\quad

We next prove the assertion about equality.
Suppose that $0<s<1$ and
\[
 y=H(s)=\min\{1-\beta_ns,c(1-s)\}.
\]
In particular $y<1$.
The inequalities $y\le1-\beta_ns$ and $y\le c(1-s)$ imply
\[
 F(a_-)=I_n\left(\frac{1-y}{s}\right)\le2^{-n},
 \qquad
 F(a_+)=2^{-n}\left(\frac{y}{c(1-s)}\right)^n\le2^{-n}.
\]
Consequently, for a parameter $a$ satisfying
\eqref{eq:support-domination}, we have the full chain
\begin{equation}\label{eq:equality-chain}
 2^{-n}
 \le n\int_0^1t^{n-1}H(t)^n\,\mathrm dt
 \le F(a)
 \le\max\{F(a_-),F(a_+)\}
 \le2^{-n}.
\end{equation}
Every inequality in this chain is therefore an equality.
If $a_-<a<a_+$, the strict part of
\eqref{eq:F-endpoint-bound} would instead give
\[
 F(a)<(1-\lambda)F(a_-)+\lambda F(a_+)\le2^{-n},
\]
which contradicts \eqref{eq:equality-chain}.
Hence the supporting parameter must be $a=a_-$ or $a=a_+$.

The equality
$n\int_0^1t^{n-1}H(t)^n\,\mathrm dt=F(a)$
in \eqref{eq:equality-chain} gives
\begin{equation}\label{eq:zero-integral}
 n\int_0^1t^{n-1}\bigl(h_a(t)^n-H(t)^n\bigr)\,\mathrm dt=0.
\end{equation}
The difference $h_a^n-H^n$ is continuous and nonnegative on
$[0,1]$. If it was positive at some point, it would be positive
on an interval contained in $(0,1)$, where $t^{n-1}>0$, so
the integral in \eqref{eq:zero-integral} would then be strictly
positive. Therefore, $h_a(t)^n=H(t)^n$ for all $t\in[0,1]$,
and nonnegativity implies that $H=h_a$.

If $a=a_-$, then $F(a_-)=2^{-n}$ and \eqref{eq:F-left}
imply
\[
 I_n(a_-)=2^{-n}=I_n(\beta_n).
\]
Strict monotonicity of $I_n$ gives $a_-=\beta_n$.
Using \eqref{eq:endpoint-functions}, we conclude that
$H(t)=h_-(t)=1-\beta_nt$.
If $a=a_+$, then \eqref{eq:F-right} and $F(a_+)=2^{-n}$
give
\[
 \left(\frac{y}{c(1-s)}\right)^n=1.
\]
Since $y>0$, this is equivalent to $y/(1-s)=c$, and hence
$H(t)=h_+(t)=c(1-t)$.
The equality in \eqref{eq:scalar-mass} has already been obtained
in \eqref{eq:equality-chain}.

Finally, the affine functions $H(t)=1-\beta_nt$ and
$H(t)=c(1-t)$ are continuous, nonnegative and concave on $[0,1]$.
Their values at zero are $1$ and $c>1$, respectively, and
\[
 \begin{aligned}
 n\int_0^1t^{n-1}(1-\beta_nt)^n\,\mathrm dt
 &=I_n(\beta_n)=2^{-n},\\
 n\int_0^1t^{n-1}\bigl(c(1-t)\bigr)^n\,\mathrm dt
 &=\frac{c^n}{\binom{2n}{n}}=2^{-n}.
 \end{aligned}
\]
Thus both functions satisfy all the hypotheses.
At a fixed $s$, if $1-\beta_ns\le c(1-s)$, choosing
$H(t)=1-\beta_nt$ gives equality in \eqref{eq:scalar-bound}.
If $c(1-s)<1-\beta_ns$, choosing $H(t)=c(1-t)$ gives equality
instead. Therefore no larger pointwise lower bound can hold
for all functions satisfying the hypotheses of the theorem.
\end{proof}

\subsection{Completion of the proof}
We now complete the proof of Theorem \ref{thm:main}.

\begin{proof}[Proof of Theorem~\ref{thm:main}]
For each pair $K,L$, apply Theorem~\ref{lem:scalar} to the function
$H$ in Lemma~\ref{lem:volume-profile}. This gives
\[
 \frac{|K+_{s^n}L|^{1/n}}{|K|^{1/n}+|L|^{1/n}}
 \ge\min\{1-\beta_ns,c(1-s)\}.
\]
Taking the infimum over $K,L$ and setting $s=\theta^{1/n}$
proves \eqref{eq:new-bound}.

Define
\[
 s_*=\frac{c-1}{c-\beta_n}.
\]
Since $c>1$ and $0<\beta_n<1$, we have $0<s_*<1$.
The difference between the two affine expressions is
\[
 (1-\beta_ns)-c(1-s)
 =1-c+(c-\beta_n)s=(c-\beta_n)(s-s_*).
\]
Thus $c(1-s)$ is the smaller expression precisely when $s\ge s_*$.
For these $s$, our lower bound and \eqref{eq:simplex-upper} give
\[
 c(1-s)\le\varphi_n(s^n)\le c(1-s).
\]
Consequently $\varphi_n(s^n)=c(1-s)$ on $[s_*,1]$.
Substituting the value of $c$ into $s_*^n$ gives
\eqref{eq:exact-value}--\eqref{eq:sharp-range}.

Suppose that equality in \eqref{eq:BM} holds for a pair $K,L$
and a parameter $\theta=s^n$ with $s_*\le s<1$.
Since $\varphi_n(s^n)=c(1-s)$, this says
\[
 H(s)=c(1-s)=\min\{1-\beta_ns,c(1-s)\}.
\]
The equality statement of Theorem~\ref{lem:scalar} therefore gives
one of the two affine profiles and
\[
 n\int_0^1t^{n-1}H(t)^n\,\mathrm dt=2^{-n}.
\]
By Lemma~\ref{lem:volume-profile}, there is a vector $z_0$
such that $K=z_0-L$. Replacing $L$ by $L-z_0$ changes each
convolution body by a translation and leaves its volume unchanged.
Therefore, we may assume $K=-L$.

For this pair, $M(K,-K)=|K|$, and
\[
 K+_1(-K)=\{z:|K\cap(z+K)|=|K|\}.
\]
If $z\in K+_1(-K)$, then $K\cap(z+K)$ has the same
volume as both $K$ and $z+K$. Since a proper inclusion
between convex bodies gives a strict inequality of
volumes, $K=z+K$.
For any $x\in K$, this identity gives $x+mz\in K$ for
every integer $m\ge0$. Since $K$ is bounded, we must have
$z=0$. Conversely, $z=0$ gives the full intersection.
Consequently,
\[
 L_1=K+_1(-K)=\{0\},
 \qquad H(1)=0.
\] 
Thus, the possibility $H(t)=1-\beta_nt$ is excluded because
$1-\beta_n>0$. This also excludes that possibility when
$s=s_*$, where the two affine expressions coincide.

It follows that $H(t)=c(1-t)$ on $[0,1]$.
At $t=0$, this gives
\[
 \frac{|K-K|^{1/n}}{2|K|^{1/n}}=c,
 \qquad
 |K-K|=(2c)^n|K|=\binom{2n}{n}|K|.
\]
The equality characterization in Rogers--Sephards inequality \eqref{eq:RS} now shows that
$K$ is a simplex.
Conversely, if $K$ and $-L$ are translates of the same simplex,
translation invariance and \eqref{eq:simplex-levels} give
$H(s)=c(1-s)$, so equality holds in the stated interval.
The restriction $\theta<1$ in the equality statement is necessary:
for example, every pair $K,-K$ has $|K+_1(-K)|=0$,
whether or not $K$ is a simplex.
\end{proof}

\begin{rmk}
For every pair of convex bodies $K,L$, the function $H$ in
Lemma~\ref{lem:volume-profile} satisfies
\[
 H(s)>1-\beta_n s,
\]
for $0<s<s_*$, where
\[
 s_*=\frac{c-1}{c-\beta_n}.
 \]
Since $1-\beta_n s<c(1-s)$ on this interval,
equality would imply, by Theorem~\ref{lem:scalar},
that $H(t)=1-\beta_n t$ for all $t\in[0,1]$,
together with equality in the last inequality of equation
\eqref{eq:profile-mass}.
Lemma~\ref{lem:volume-profile} would then imply that
$K$ and $-L$ are translates. As shown in the preceding
proof, this gives $H(1)=0$, contradicting
$H(1)=1-\beta_n>0$.

This excludes attainment of the first branch by any pair
of convex bodies for $0<s<s_*$. However, this does not necessarily
establish the strict inequality
\[
 \varphi_n(s^n)>1-\beta_n s,
 \qquad 0<s<s_*,
\]
since the infimum over all pairs of convex bodies may still equal
$1-\beta_ns$
\end{rmk}

\section{Other bounds}
\label{sec:comparison}

After the substitution $u=xt$, the integral defining
$I_n(x)$ is taken over $[0,x]$. Extending this interval
to $[0,1]$ gives an upper bound for $I_n(x)$, but leads
to a weaker lower bound for $H(s)$. Indeed, for $0<x<1$,
\eqref{eq:beta-value} gives
\begin{align*}
 I_n(x)
 &=\frac{n}{x^n}\int_0^x u^{n-1}(1-u)^n\,\mathrm du\\
 &<\frac{n}{x^n}\int_0^1 u^{n-1}(1-u)^n\,\mathrm du
 =2^{-n}(cx)^{-n}.
\end{align*}
The inequality is strict because the integrand is
positive on $(x,1)$. In particular,
$I_n(1/c)<2^{-n}=I_n(\beta_n)$.
Since $I_n$ is strictly decreasing, we obtain
\begin{equation}\label{eq:beta-comparison}
 0<\beta_n<\frac1c=2\binom{2n}{n}^{-1/n}.
\end{equation}

If we used the upper bound $2^{-n}(cx)^{-n}$ in place
of $I_n(x)$ in \eqref{eq:F-left}, with $x=(1-y)/s$,
we would obtain $F(a_-)<2^{-n}$ whenever $y<1-s/c$.
The argument in the proof of Theorem~\ref{lem:scalar}
would then give the weaker bound
\[
 H(s)\ge\min\{1-s/c,c(1-s)\}.
\]
By \eqref{eq:beta-comparison}, the first expression
in our bound satisfies
\[
 1-\beta_ns>1-s/c,\qquad s>0.
\]
The weaker bound equals $c(1-s)$ for $s\ge c/(c+1)$,
whereas our bound already equals $c(1-s)$ for
$s\ge s_*$, where
\[
 s_*=\frac{c-1}{c-\beta_n}
 <\frac{c}{c+1}<\frac23.
\]
Thus our bound determines $\varphi_n$ on $[s_*^n,1]$,
rather than on the smaller interval
$[(c/(c+1))^n,1]$ obtained by replacing $I_n(x)$
with the upper bound above.

\subsection{Comparison with the previous known bound}

In the non sharp regime, authors in \cite{AGMG} obtained the lower bound
$1-(4-c)s/3$. The weaker function $1-s/c$ is already strictly
larger than this expression for $s>0$, since
\begin{align*}
 \left(1-\frac{s}{c}\right)-\left(1-\frac{4-c}{3}s\right)
 &=\frac{(c-1)(3-c)}{3c}\,s,\\
 c(1-s)-\left(1-\frac{4-c}{3}s\right)
 &=(c-1)\left(1-\frac43s\right).
\end{align*}
For $0<s<3/4$, both right-hand sides are positive because
$1<c<2$. Since $1-\beta_ns>1-s/c$, each of the two expressions
giving the lower bound of $\varphi_n$ in Theorem \ref{thm:main} are strictly larger than
$1-(4-c)s/3$, and so is their minimum.
This proves the strict improvement over \eqref{eq:AGMG}
for $0<\theta<(3/4)^n$.

\subsection{An upper bound from Euclidean balls}

The example of Euclidean balls gives an upper bound for
$\varphi_n$ and restricts how far the interval on which the
simplex gives the optimal value can extend towards zero.
More precisely,
\begin{equation}\label{eq:ball-upper}
 \varphi_n(s^n)
 \le \min\{c(1-s),\sqrt{1-s^2}\},
 \qquad 0\le s\le1,
\end{equation}
and
\begin{equation}\label{eq:ball-obstruction}
 \varphi_n(s^n)<c(1-s),
 \qquad
 0<s\le\frac{c^2-1}{c^2+1}.
\end{equation}
Thus the simplex formula is not optimal at any parameter
in \eqref{eq:ball-obstruction}. In particular, if
\eqref{eq:exact-value} holds for every
$\theta\in[\sigma^n,1]$, with $0\le\sigma\le s_*$, then
\[
 \frac{c^2-1}{c^2+1}<\sigma\le s_*.
\]
Our bounds do not determine how far the exact interval
can be extended within this range.

In the rescaled variable $s=\theta^{1/n}$, the endpoints
of this range tend to $3/5$ and $2/3$, respectively,
as $n\to\infty$.

To prove these assertions, let $B$ be the Euclidean unit
ball and fix $0<s<1$. Since $M(B,B)=|B|$, the corresponding
level set is
\[
 L_s=B+_{s^n}B
 =\{z\in2B:|B\cap(z-B)|\ge s^n|B|\}.
\]
By convexity and rotational symmetry, $L_s$ is a ball
centered at the origin. Continuity of the covariogram
shows that its radius lies strictly between zero and two.
Writing $L_s=2rB$, with $0<r<1$, we obtain
\[
 \varphi_n(s^n)
 \le\frac{|L_s|^{1/n}}{2|B|^{1/n}}=r.
\]
For any unit vector $u$, the boundary point $2ru$ satisfies
\[
 |B\cap(2ru-B)|=s^n|B|.
\]
Indeed, if the volume were larger, continuity would put
points slightly farther along the same ray in $L_s$,
contradicting its radius.

Translate this intersection by $-ru$. Since $B=-B$,
the translated set is $(B-ru)\cap(B+ru)$.
For every point $x$ in this set,
\[
 \|x+ru\|_2^2\le1,\qquad
 \|x-ru\|_2^2\le1.
\]
Adding these inequalities gives
$\|x\|_2^2+r^2\le1$, and hence
\[
 (B-ru)\cap(B+ru)\subseteq\sqrt{1-r^2}\,B.
\]

The inclusion is strict in volume. Indeed, a point $tu$
with $1-r<t<\sqrt{1-r^2}$ lies in the interior of the
ball on the right and outside the intersection.
A sufficiently small open ball centered at $tu$ has
the same property. Therefore,
\[
 s^n
 =\frac{|B\cap(2ru-B)|}{|B|}
 <(1-r^2)^{n/2},
\]
which yields
\[
 \varphi_n(s^n)\le r<\sqrt{1-s^2}.
\]
Combining this estimate with \eqref{eq:simplex-upper}
and the endpoint values $\varphi_n(0)=1$ and
$\varphi_n(1)=0$ proves \eqref{eq:ball-upper}.

For $0<s<1$, comparison with the simplex gives
\begin{align*}
 \sqrt{1-s^2}\le c(1-s)
 &\quad\Longleftrightarrow\quad
 1+s\le c^2(1-s)\\
 &\quad\Longleftrightarrow\quad
 s\le\frac{c^2-1}{c^2+1}.
\end{align*}
Combining this with the strict upper bound just proved
gives \eqref{eq:ball-obstruction}, including its right
endpoint.

Finally,
\[
 \frac{4^n}{2n+1}\le\binom{2n}{n}<4^n
 \quad\Longrightarrow\quad c\longrightarrow2.
\]
Also, $(1-t/2)^n>2^{-n}$ for $0<t<1$, so
$I_n(1/2)>2^{-n}=I_n(\beta_n)$. The strict decrease
of $I_n$, together with the bound of $\beta _n$ in \eqref{eq:beta-comparison},
therefore gives $1/2<\beta_n<1/c$. It follows that
\[
 \beta_n\longrightarrow\frac12,\qquad
 \frac{c^2-1}{c^2+1}\longrightarrow\frac35,
 \qquad
 s_*=\frac{c-1}{c-\beta_n}\longrightarrow\frac23.
\]

\quad

\quad

\section*{Acknowledgements}

The author thanks Prof. David Alonso Gutiérrez for his careful reading the manuscript
and for helpful suggestions on its presentation.

\end{document}